\documentclass[a4paper,psamsfonts,reqno]{amsart}
\usepackage{amssymb,amscd,amsmath}

\usepackage[active]{srcltx}
\usepackage{upref}
\usepackage[american]{babel}
\usepackage[all]{xy}
\usepackage{fancyhdr}
\usepackage{enumerate}

\theoremstyle{plain}
\newtheorem{theo}{Theorem}

\newtheorem{prop}[theo]{Proposition}
\newtheorem{lem}[theo]{Lemma}
\newtheorem{coro}[theo]{Corollary}

 \theoremstyle{definition}

 \newtheorem{ex}[theo]{Example}

 \theoremstyle{remark}

\numberwithin{equation}{section}

\newcommand{\naturals}{\mathbb{N}}
\newcommand{\integers}{\mathbb{Z}}

\DeclareMathOperator{\supp}{supp} 
 \DeclareMathOperator{\Aut}{Aut}

\usepackage[pdftex]{hyperref}

\begin{document}

\title{Free group algebras in Malcev-Neumann skew fields of fractions}

\author[J. S\'{a}nchez]{Javier S\'{a}nchez}
\address{Department of Mathematics - IME, University of S\~ao Paulo,
Caixa Postal 66281, S\~ao Paulo, SP, 05314-970, Brazil}
\email{jsanchez@ime.usp.br}
\thanks{Supported by FAPESP Processo n\textordmasculine\ 2009/50886-0,
by the DGI and the European Regional Development Fund, jointly,
through Project MTM2008-06201-C02-01, and by the Comissionat per
Universitats i Recerca of the Generalitat de Catalunya, Project 2009
SGR 1389.}



\date{12 July 2011}

\begin{abstract}
Let $K$ be a skew field and $(G,<)$ an ordered group. We show that
the skew field generated by the group ring $K[G]$ inside the
Malcev-Neumann series ring $K((G;<))$ contains noncommutative free
group algebras.
\end{abstract}

\maketitle

\section*{Introduction}

There are two conjectures on the existence of free objects in skew
fields that have attracted much attention. The first one was made by
A. Lichtman~\cite{Lichtmanonsubgroupsof}:

\begin{enumerate}[(G)]
\item If a skew field $D$ is not commutative (i.e. not a field),
then $D\setminus\{0\}$ contains a noncyclic free group.
\end{enumerate}
An affirmative answer has been given when the center of $D$ is
uncountable~\cite{chibafreegroupsinsidedivisionrings} and when $D$
is finite-dimensional over its center
\cite{Goncalvesfreegroupsinsubnormal}. The second conjecture was
formulated in~\cite{Makaronfreesubobjects}:
\begin{enumerate}[(A)]
\item Let $D$ be a skew field with center $Z$. If $D$ is finitely
generated as skew field over $Z$, then $D$ contains a noncommutative
free $Z$-algebra.
\end{enumerate}
Some important results have been obtained in \cite{Makar1},
\cite{Makar2}, \cite{Lichtmanfreesubalgebrasenvelopingalgebras},
\cite{BellRogalskifreesubalgebrasoreextensions} and other papers. In
many examples in which conjecture (A) holds, $D$ in fact contains a
noncommutative free group $Z$-algebra
\cite{FigueiredoGoncalvesShirvani}. For example, this always happens
if the center of $D$ is uncountable~\cite{GoncalvesShirvani}.
Therefore it makes sense to consider the following unifying
conjecture:
\begin{enumerate}[(GA)]
\item Let $D$ be a skew field with center $Z$.
If $D$ is  finitely generated as a skew field over $Z$ and $D$ is
infinite dimensional over $Z$, then $D$ contains a noncommutative
free group $Z$-algebra.
\end{enumerate}
For more details on these and related conjectures the reader is
referred to \cite{GoncalvesShirvaniSurvey}.

\medskip

We deal with conjectures (A) and (GA) for an important class of skew
fields obtained from a skew field $K$ and an ordered group $(G,<)$,
i.e.  $G$ is a group  and  $<$ is a total order  in $G$ such that
the relation $x<y$ implies that $zx<zy$ and $xz<yz$ for all
$x,y,z\in G$. The Malcev-Neumann series ring $K((G;<))$ is defined
to be the set of all formal series $f=\sum_{x\in G}xa_x$, with
$a_x\in K$, whose \emph{support}  is a well-ordered subset of
$(G,<)$ where $\supp (f)=\{x\in G\mid a_x\neq 0\}$. The operations
of $K$ and $G$, together with the relations $ax=xa$, for all $a\in
K$ and $x\in G$, induce a multiplication in $K((G;<))$. With this
product and componentwise addition, $K((G;<))$ is a skew field
\cite{Malcev}, \cite{Neumann}. Clearly $K((G;<))$ contains the group
ring $K[G]$. We denote by $K(G)$ the skew subfield of $K((G;<))$
generated by $K[G]$. It was proved in \cite{Chibafreefields} that
$K(G)$ is infinite dimensional over its center when $G$ is not
commutative.

Our main result is Theorem~\ref{theo:mainresult}. It states that if
$(G,<)$ is not commutative, then $K(G)$ satisfies conjecture (GA).
As a consequence one obtains that any skew field generated by a
group ring $K[G]$ with $G$ a noncommutative orderable group and $K$
a skew field satisfies conjecture (A), see
Corollary~\ref{coro:maincorollary}. We remark that our methods
exhibit the generators of the free group algebra when the order of
the group is well known.

\medskip

Although $K(G)$ does not depend on the structure of the ordered
group $(G,<)$ \cite{Hughes}, our proof of the main result strongly
depends on it. Our proof also needs of  richer structures than group
rings and the Malcev-Neumann series rings already defined, namely
crossed products and the Malcev-Neumann series they determine. These
and related concepts are reviewed in the first section.

Section~\ref{sec:usefulresults} deals with results that will be used
to reduce the problem of finding free group algebras inside skew
fields. The first one is a generalization of the fact that finding
free (group) algebras over the prime subfield  equivalent to finding
free (group) algebras over any central subfield
\cite[Lemma~1]{MakarMalcolmson91}. The second one states that if our
orderable group $G$ contains a free monoid, then $K(G)$ contains a
free group algebra. We also give some results describing the
behavior of Malcev-Neumann series rings that will be used later.
They may be well known, but we provide a proof because we find them
of independent interest and we have not found them in the
literature.

The core of the paper is Section~\ref{sec:proofofmainresult}. Here
we give a proof of our main result. It is divided in three
subsections according to what type of ordered group $(G,<)$ is. The
proof for ordered groups of type~1 consists in showing that the
problem can be reduced to torsion-free nilpotent groups. For these
groups the result was proved in \cite{FigueiredoGoncalvesShirvani}.
For ordered groups of type 2 and 3, we are able to find free monoids
inside $G$.

In Section~\ref{sec:someconsequences} we give some consequences from
the main result and the methods we have developed for its proof.


\section{Preliminaries}\label{sec:preliminaries}

In this section we present the concepts and notation that will be
used throughout the paper.

A \emph{skew field} is a nonzero associative ring with identity
element and such that every nonzero element has an inverse.

If $R$ is a ring, a \emph{skew field of fractions} of $R$ is a skew
field $D$ together with a ring embedding $R\hookrightarrow D$ such
that $D$ is generated as a skew field by (the image of) $R$.

\subsection{Ordered groups}

We  recall  standard material on ordered groups. For further details
the reader is referred to any monograph on the subject such as
\cite{Fuchs}. We will sometimes use the results reviewed here
without any further reference.

A group $G$ is  \emph{orderable} if there exists a total order $<$
of G
 such that, for all  $x, y, z \in G$,
\begin{equation}\label{eq:orderedcondition}
 x < y\ \ \textrm{ implies that }\ \ zx < zy \textrm{ and } xz < yz.
\end{equation}
In this event we say that $(G,<)$ is an \emph{ordered group}.

Let $(G,<)$ be an ordered group. A subgroup $H$ of $G$ is
\emph{convex} if for any $x,z\in H$ and $y\in G$, the inequalities
$x< y< z$ imply that $y\in H$. If a convex subgroup $N$ is normal in
$G$, then $G/N$ can be ordered in a natural way: given $x,y\in G$
such that $xN\neq yN$, define $xN \prec yN$ if $x<y$. Moreover there
is a natural bijection between the convex subgroups of $(G,<)$ that
contain $N$ and the convex subgroups of $(G/N,\prec)$.

An element $x\in G$ is \emph{positive} if $1<x$.

The ordered group $(G,<)$ is \emph{archimedean} if for any positive
elements $x,y\in G$  there is a natural number $n$ such that
$x<y^n$, equivalently, $(G,<)$ has no other convex subgroups other
than $\{1\}$ and $G$. An archimedean group is order isomorphic to a
subgroup of the additive group of the real numbers with the natural
order.

If $A,B$ are subgroups of the additive group of the real numbers
with the natural order and $\varphi\colon A\rightarrow B$ is a
\emph{morphism of ordered groups} (i.e.  a morphism of groups such
that $x<y$ implies $\varphi(x)\leq \varphi(y)$), then there exists a
real number $r\geq 0$ such that $\varphi(x)=rx$ for all $x\in A$.

The set $\Gamma$ of convex subgroups of $(G,<)$ satisfies the
following properties:
\begin{enumerate}[$-$]
\item $\Gamma$ is a chain with respect to inclusion, that is, if $H_1,H_2\in
\Gamma$, then either $H_1\subseteq H_2$ or $H_2\subseteq H_1$.
\item The arbitrary union and intersection of elements in $\Gamma$
belongs to $\Gamma$
\item $\Gamma$ is \emph{infrainvariant}, that is, $xHx^{-1}\in \Gamma$ for
each $H\in \Gamma$ and $x\in G$.
\end{enumerate}

We say that a pair $(N,H)$ is a \emph{convex jump} if
$N,H\in\Gamma$, $N\lneqq H$ and no subgroup in $\Gamma$ lies
properly between $N$ and $H$. It is known that   $N\lhd H$ for any
convex jump $(N,H)$. Hence $H/N$ is order isomorphic to a subgroup
of the additive group of the reals with the natural order.

Let  $x\in G$. The \emph{convex jump determined by $x$} is the
convex jump $(N_x,H_x)$, where $N_x$ is the union of all convex
subgroups of $(G,<)$ that do not contain $x$, and $H_x$ is the
intersection of all convex subgroups of $(G,<)$ that contain $x$.

\subsection{Crossed product group rings}\label{sec:crossedproducts}

The following definitions and comments are from
\cite[Chapter~1]{Passman1}.

Let $R$ be a ring and let $G$ be a group. A \emph{crossed product}
 of $G$ over $R$ is an associative ring $R[\bar{G};\sigma,\tau]$
which contains $R$ and has as an $R$-basis the set $\bar{G}$, a copy
of $G$. Thus each element of $R[\bar{G};\sigma,\tau]$ is uniquely
expressed as a finite sum $\Sigma_{x\in G}\bar{x}r_x$, with $r_x\in
R$. Addition is componentwise and multiplication is determined by
the two rules below. For $x, y \in G$ we have a \emph{twisting}
$$\bar{x}\bar{y} = \overline{xy}\tau(x,y)$$ where $\tau\colon G
\times G \rightarrow U(R)$, the group of units of $R$. Furthermore
for $x\in G$ and $r\in R$ we have an \emph{action}
$$r\bar{x}=\bar{x}r^{\sigma(x)}$$
where $\sigma\colon G\rightarrow \Aut(R)$ and $r^{\sigma(x)}$
denotes the image of $r$ by $\sigma(x)$.

If $d\colon G\rightarrow U(R)$ assigns to each element $x\in G$ a
unit $d_x$ of $R$, then $\tilde{G} = \{ \tilde{x} = \bar{x}d_x\mid
x\in G \}$ is another $R$-basis for $R[\bar{G};\sigma,\tau]$.
Moreover there exist maps $\sigma'\colon G\rightarrow \Aut(R)$ and
$\tau'\colon G\times G\rightarrow U(R)$ such that
$R[\bar{G};\sigma,\tau]=R[\tilde{G};\sigma',\tau']$. We call this a
\emph{diagonal change of basis}. Via a diagonal change of basis if
necessary, we will assume that $1_{R[\bar{G};\sigma,\tau]}=
\bar{1}$. The embedding of $R$ into $R[\bar{G};\sigma,\tau]$ is then
given by $r\mapsto r\bar{1}$. Note that $\bar{G}$ acts on both
$R[\bar{G};\sigma,\tau]$ and $R$ by conjugation and that for $x\in
G,$ $r \in R$ we have $\bar{x}^{-1}r\bar{x} = r^{\sigma(x)}$.

If there is no action or twisting, that is, if $\sigma(x)$ is the
identity and $\tau( x , y )= 1$ for all $x,y \in G$, then
$R[\bar{G};\sigma,\tau]= R[G]$ is the usual \emph{group ring}.

If $N$ is a subgroup of $G$, the crossed product
$R[\bar{N};\sigma_{\mid N},\tau_{\mid N\times N}]$ embeds in
$R[\bar{G};\sigma,\tau]$, and we will denote it as
$R[\bar{N};\sigma,\tau]$.

Suppose that $N$ is a normal subgroup of $G$ and let
$S=R[\bar{N};\sigma,\tau]$. Then $R[\bar{G};\sigma,\tau]$ can be
seen as a crossed product of $G/N$ over $S$. More precisely
$R[\bar{G};\sigma,\tau]=S[\overline{G/N};\tilde{\sigma},\tilde{\tau}]$,
where $\overline{G/N}$,  $\tilde{\sigma}$ and $\tilde{\tau}$ are
defined as follows.

For each $1\neq\alpha\in G/N,$ let $x_\alpha\in G$ be a fixed
representative of the class $\alpha$. For the class $1\in G/N$
choose $x_1=1\in G$. Set $\bar{\alpha}=x_\alpha$ for all $\alpha\in
G/N$. Then $G=\dot{\bigcup}\bar{\alpha}N$. Define
$\overline{G/N}=\{\bar{\alpha}\mid \alpha\in G/N\}\subseteq\bar{G}$.
This shows that
$R[\bar{G};\sigma,\tau]=\bigoplus\limits_{\alpha}\bar{\alpha}S$.
Therefore $\overline{G/N}$ is an $S$-basis for
$R[\bar{G};\alpha,\tau]$. Moreover, we have the maps:
$$\begin{array}{c @{} c @{} c @{} c @{} c @{} c @{} l }
\tilde{\sigma}  & \colon G/N    & \longrightarrow &
\multicolumn{4}{l}{\textrm{Aut}(S)}
      \\

               &  \alpha   & \longmapsto     & \hspace{0.5cm}\tilde{\sigma}(\alpha)

 & \colon S & \longrightarrow & \hspace{0.25cm} S \\
       &      &         &                 &
  \hspace{0.2cm}y & \longmapsto     & \hspace{0.1cm} \bar{x}_{\alpha}^{-1}y{\bar{x}_{\alpha}} \\
\end{array}$$
$$ \begin{array}{c @{} c @{} c @{} l} \tilde{\tau} & \colon G/N
\times G/N & \longrightarrow &
\hspace{0.1cm} U(S) \\
 & (\alpha,\beta) & \longmapsto & \hspace{0.1cm}
 \tilde{\tau}(\alpha,\beta)=\bar{n}_{\alpha\beta}\tau(x_{\alpha\beta},n_{\alpha\beta})^{-1}\tau(x_{\alpha},x_{\beta})
\end{array}$$
where $n_{\alpha\beta}$ is the unique element in $N$ such that
$x_{\alpha\beta}n_{\alpha\beta}=x_\alpha x_\beta$. Then
$(\sum\limits_{n\in
N}\bar{n}r_n)\bar{\alpha}=\bar{\alpha}{(\sum\limits_{n\in
N}\bar{n}r_n)}^{\tilde{\sigma}(\alpha)}$ and $\bar{\alpha}
\bar{\beta}=\overline{\alpha\beta}\tilde{\tau}(\alpha,\beta)$, as
desired.

Notice that the assumption  ($x_1=1$) implies that $S$ embeds in
$R[\bar{G};\sigma,\tau]$ via $\bar{n}r\mapsto 1\cdot\bar{n}r.$

\subsection{Malcev-Neumann series rings}

Let $R$ be a ring, and let $(G,<)$ be an ordered group. Form a
crossed product  $R[\bar{G};\sigma,\tau].$ We define a new ring,
denoted $R((\bar{G};\sigma,\tau,<))$ and called \emph{Malcev-Neumann
series ring}, in which $R[\bar{G};\sigma,\tau]$ embeds. As a set
$$R((\bar{G};\sigma,\tau,<))=\Big\{f=\sum_{x\in G}\bar{x}a_x\mid a_x\in R,\ \supp(f) \textrm{ is
well ordered}\Big\},$$ where $\supp(f)=\{x\in G\mid a_x\neq 0\}$.

The addition and product are defined extending the ones in
$R[\bar{G};\sigma,\tau]$. Thus, given $f=\sum\limits_{x\in G}\bar x
a_x$ and $g=\sum\limits_{x\in G}\bar x b_x$ in
$R((\bar{G};\sigma,\tau,<))$, the sum is defined by
$$f+g=\sum_{x\in G}\bar x (a_x+b_x),$$ and multiplication by
$$fg=\sum_{x\in
G}\bar x\Big(\sum_{yz=x}\tau(y,z)a_y^{\sigma(z)}b_z\Big).$$ When the
crossed product is a group ring $R[G]$, its Malcev-Neumann series
ring will be denoted by $R((G;<))$.

Let $f\in R((\bar{G};\sigma,\tau,<))$ and $x_0=\min\supp f.$ If the
coefficient of $x_0$ is an invertible element of $R$, then $f$ is
invertible. Hence, if $R$ is a skew field, then
$R((\bar{G};\sigma,\tau,<))$ is a skew field \cite{Malcev},
\cite{Neumann}.

If $K$ is a skew field, the skew field of fractions of
$K[\bar{G};\sigma,\tau]$ inside $K((\bar{G};\sigma,\tau,<))$ will be
called the \emph{Malcev-Neumann skew field of fractions} of
$K[\bar{G};\sigma,\tau]$ and denoted by $K(\bar{G};\sigma,\tau)$. It
is important to observe the following. For a subgroup $H$ of $G$,
$K((\bar{H};\sigma,\tau))$ and $K(\bar{H};\sigma,\tau)$ can be seen
as skew subfields of $K((\bar{G};\sigma,\tau))$ and
$K(\bar{G};\sigma,\tau)$, respectively, in the natural way. In the
case of the group ring $K[G]$, the Malcev-Neumann skew field of
fractions is denoted by $K(G)$. We remark that
$K(\bar{G};\sigma,\tau)$ does not depend on the order $<$ of $G$,
see \cite{Hughes}.


\section{Technical results}\label{sec:usefulresults}

This section is devoted to prove some results  that will be used to
prove our main result, but they are interesting in themselves too.

The next lemma is a generalization of known results that we will
obtain as a corollary. It reduces the problem of finding free
(group) algebras over the center of a skew field. The proof is very
similar to the one of the original result
\cite[Lemma~1]{MakarMalcolmson91}.

We recall that a monoid is orderable if there exists a total order
that satisfies condition \eqref{eq:orderedcondition}. Hence ordered
groups are ordered monoids.

\begin{lem}\label{lem:extfreegroupalgebra}
Let $R$ be a ring with prime subring  $P$ (i.e. generated by $1$).
Let $M$ be a submonoid of $U(R)$ and let $C$ be a subring of $R$.
Suppose that the following conditions hold true
\begin{enumerate}[\rm(1)]
\item The algebra generated by $P$ and $M$ is the monoid
algebra $P[M]$.
\item For each $c\in C$ and $f\in P[M]$, the equality $cf=0$ implies
that $c=0$ or $f=0$.
\item The elements of $M$ commute with the elements of $C$.
\item $M$ is an orderable monoid.
\item $M$ embeds in some group $H$ such that $H(M)$, the subgroup of
$H$ generated by $M$, has trivial center.
\end{enumerate}
Then the subring of $R$ generated by $C$ and $M$ is the monoid ring
$C[M]$.
\end{lem}

\begin{proof}
Consider the multiplication map $\mu\colon
C\otimes_{P}P[M]\rightarrow R$. By (3), to prove that $C[M]$ is
contained in $R$, it is enough to show that $\mu$ is an injective
map.

Suppose that $\mu(\sum_{i=1}^n c_i\otimes
f_i)=\sum_{i=1}^nc_if_i=0$, where $n$ is the minimal number of
nonzero summands. If $n=1$, then $c_1f_1=0$ implies that $c_1\otimes
f_1=0$ because of (2). If $n>1$, let $g\in P[M]$ be arbitrary. By
(3), the elements of $C$ and $P[M]$ commute. Hence $\mu(\sum_{i=2}^n
c_i\otimes(f_igf_1-f_1gf_i))=0$. Since this element of the kernel
has $n-1$ summands we must have
\begin{equation}\label{eq:monoidalgebra1}
f_igf_1=f_1gf_i, \ i\in\{1,\dotsc,n\},\ g\in P[M].
\end{equation}

Fix an ordering $<$ of the monoid $M$ such that $(M,<)$ is an
ordered monoid. Let $f_{i0},g_0$ be the greatest elements of $\supp
f_i$ and $\supp g$ respectively. Then \eqref{eq:monoidalgebra1}
implies that
\begin{equation}\label{eq:monoidalgebra2}
f_{i0}mf_{10}=f_{10}mf_{i0}, \ i\in\{1,\dotsc,n\},\ m\in M.
\end{equation}
Now we look at \eqref{eq:monoidalgebra2} as an equality in $H(M)$.
Letting  $m=1$, $f_{10}^{-1}f_{i0}=f_{i0}f_{10}^{-1}$ for
$i\in\{1,\dotsc,n\}$. This and \eqref{eq:monoidalgebra2} tells us
that $f_{10}^{-1}f_{i0}$ is in the center of $H(M)$ for
$i\in\{1,\dotsc,n\}$. Hence $f_{10}=f_{i0}$ for $i\in\{1,\dotsc,n\}$
because $H(M)$ has trivial center.
 Therefore, for some $m_i\in P$,  the greatest element
in $\supp (f_i-m_if_1)$ is smaller than $f_{i0}$, but we still have
$(f_i-m_if_1)gf_1=f_1g(f_i-m_if_1)$ for each $g$ and $i=1,\dotsc,
n$. Thus $f_i=m_if_1$ for each $i$ and $\sum_{i=1}^n c_i\otimes
f_i=(\sum_{i=1}^nc_im_i)\otimes f_1$ reducing to the case of $n=1$.
This shows that $\ker \mu=0$.
\end{proof}

Suppose that $R$ is a domain (a ring without zero divisors), that
$C$ is a central subfield of $R$, and that $M$ is either the free
group or the free monoid. The free group is known to be orderable,
see for example \cite{Fuchs}. Hence all conditions in
Lemma~\ref{lem:extfreegroupalgebra} are satisfied. Thus we obtain
Lemma~\ref{lem:freegroupalgebraovercentre}. It is
\cite[Lemma~1]{MakarMalcolmson91} and
\cite[Lemma~2.1]{GoncalvesShirvani}.

\begin{lem}\label{lem:freegroupalgebraovercentre}
Let $R$ be a domain with prime subfield $P$. Let $C$ be any central
subfield of $R$. Let $M$ be a free submonoid (subgroup) of $U(R)$.
Then the algebra generated by $P$ and $M$ is the monoid (group)
algebra $P[M]$ if and only if the algebra generated by $C$ and $M$
is the monoid (group) algebra $C[M]$. \qed
\end{lem}

\medskip

 Our next result shows that there is a free group algebra
inside the Malcev-Neumann skew field of fractions provided there is
a free monoid inside the ordered group.

Let $X=\{x_i\}_{i\in I}$ be a set, $M$ the free monoid on $X$ and
$H$ the free group on $X$.  If $R$ is a ring, $R\langle\langle
X\rangle\rangle=\{\sum_{w\in M}wr_w\mid r_w\in R \textrm{ for all }
w\in M\}$ is a ring with the natural sum and the product induced
from the operations of $R$ and $M$, together with the relations
$rw=wr$, for all $r\in R$ and $w\in M$. It has been proved in
\cite[Section~2]{AraDicks} that  the group ring $R[H]$ embeds in the
power series ring via the morphism of $R$-rings $\Upsilon\colon
R[H]\rightarrow R\langle\langle X\rangle\rangle$ determined by
$\Upsilon(x_i)=1+x_i$ for all $i\in I$. For the case of
$R=\integers$ this result was proved in
\cite[Theorem~4.3]{Foxembedding} and as noted in \cite{AraDicks},
the same proof of Fox works also with any ring $R$.

\begin{lem}\label{lem:freemonoidfreegroupalgebra}
Let $(G,<)$ be an ordered group. Let $X$ be a subset of $G$
consisting of positive elements and such that $X$ is the basis of a
noncommutative free submonoid of $G$. Let $K$ be a skew field and
$K[\bar{G};\sigma,\tau]$ a crossed product. Then
$K(\bar{G};\sigma,\tau)$ contains a noncommutative free group
algebra over its center $Z$. More precisely, $\{1+\bar{x}\mid x\in
X\}$ is a basis of a free group $H$, and the free group algebra
$Z[H]$ embeds in $K(\bar{G};\sigma,\tau)$.
\end{lem}

\begin{proof}
Let $P$ be the prime subfield of $K$. For any finite subset $Y$ of
$X$, let $M_Y$ denote the free submonoid of $G$ generated by $Y$.
Note that  $(1-(\sum_{x\in Y}\bar{x}))^{-1}=\sum_{w\in
M_Y}\bar{w}\zeta_w\in K((\bar{G};\sigma,\tau,<))$ where $\zeta_w\in
K\setminus\{0\}$ for each $w\in M_Y$. Hence  $M_Y$ is a well-ordered
subset of $G$. Moreover, it is easy to prove by induction on the
lengths  that
$\bar{w}_1\zeta_{w_1}\bar{w}_2\zeta_{w_2}=\overline{w_1w_2}\zeta_{w_1w_2}$
for all $w_1,w_2\in M_Y$. Thus $Y$ is the basis of a free monoid.
The elements of $P$ commute with the elements of
$\{\bar{w}\zeta_{w}\mid w\in N\}$, we obtain that the power series
rings $P\langle\langle Y\rangle\rangle$ embeds in
$K((\bar{G};\sigma,\tau,<))$ via the natural morphism of
$P$-algebras defined by $x\mapsto \bar{x}$, $x\in Y$. If we set
$H_Y$ to be the free group on $Y$, then $P[H_Y]$ embeds in
$K((\bar{G};\sigma,\tau,<))$ via $x\mapsto 1+\bar{x}$, $x\in Y$.
Observe that $1+\bar{x}\in K(\bar{G};\sigma,\tau)$, and therefore
$P[H_Y]$ embeds in $K(\bar{G};\sigma,\tau)$. Since this can be done
for any finite set $Y$ of $X$, we obtain that in fact $P[H]$ embeds
in $K(\bar{G};\sigma,\tau)$. By
Lemma~\ref{lem:freegroupalgebraovercentre}, the free group algebra
$Z[H]$  embeds in $K(\bar{G};\sigma,\tau)$.
\end{proof}

We note that $K$ could be any ring and $Z$ the centralizer of
$\{\bar{x}\mid x\in X\}$ in
Lemma~\ref{lem:freemonoidfreegroupalgebra}

\medskip

We say that a group $G$ is \emph{Ore embeddable} if  the ring
$K[\bar{G};\sigma,\tau]$ is an Ore domain for each skew field $K$
and crossed product $K[\bar{G};\sigma,\tau]$. Examples of Ore
embeddable groups are torsion-free abelian groups and torsion-free
nilpotent groups.

\begin{lem}\label{lem:extensionscalars}
Let $R$ be a ring with a skew field of fractions $\iota\colon
R\hookrightarrow F$. Let $G$ be a group. Consider a crossed product
$R[\bar{G};\sigma,\tau]$. Suppose that the automorphism
$\sigma(x)\in\Aut(R)$ extends to a (unique) automorphism
$\varsigma(x)\in \Aut(F)$ for each $x\in G$. Then
$$\kappa\colon R[\bar{G};\sigma,\tau]\hookrightarrow F[\bar{G};\varsigma,\tau]$$
by extension of scalars, that is, $\kappa_{\mid R}=\iota$ and
$\kappa(\bar{x})=\bar{x}$.

If, moreover, $R$ is an Ore domain, and $G$ is an Ore embeddable
group, then $R[\bar{G};\sigma,\tau]$ is an Ore domain with the same
Ore skew field of fractions as $F[\bar{G};\varsigma,\tau]$.
\end{lem}

\begin{proof}
First of all we construct $F[\bar{G};\varsigma,\tau]$. Consider a
free $F$-module $T$ with basis $\bar{G}$, a copy of $G$. Then each
element of $T$ is a finite sum of the form $\sum_{x\in
G}\bar{x}a_x$, where $a_x\in F$ for each $x\in G$.
 Define the
product map $\prod\colon T\times T\rightarrow T$ by
$$\Big(\sum_{x\in G}\bar{x}a_x, \sum_{x\in
G}\bar{x}b_x\Big)\longmapsto \sum_{x\in G}\bar{x}\Big(\sum_{yz=x}
\tau(y,z)a_y^{\varsigma(z)}b_z\Big).$$ Notice that $\kappa$ defines
a right $R$-module embedding of $R[\bar{G};\sigma,\tau]$ inside $T$.
Moreover, $\prod$ restricted to (the image of)
$R[\bar{G};\sigma,\tau]$ is the product in $R[\bar{G};\sigma,\tau]$.
Now, since $\iota$ is an epimorphism of rings
\cite[Proposition~4.1.1]{Cohnskew}, the conditions of
\cite[Lemma~1.1]{Passman1} are satisfied for $\varsigma$ and $\tau$.
This shows that $T$ is a ring with product defined  by the map
$\prod$, and $T=F[\bar{G};\varsigma,\tau]$. Then clearly
$R[\bar{G};\sigma,\tau]$ embeds in $F[\bar{G};\varsigma,\tau]$ via
$\kappa$, as desired.

Suppose that $R$ is an Ore domain. By the universal property of Ore
domains, any automorphism of $R$ can be extended to its Ore field of
fractions $F$. Since $G$ is Ore embeddable,
$F[\bar{G};\varsigma,\tau]$ has an Ore skew field of fractions.
Hence $R[\bar{G};\sigma,\tau]$ embeds in
$Q_{cl}(F[\bar{G};\varsigma,\tau])$, the Ore skew field of fractions
of $F[\bar{G};\varsigma,\tau]$. A classical argument shows that the
elements of $Q_{cl}(F[\bar{G};\varsigma,\tau])$ are of the form
$g^{-1}f$ ($fg^{-1}$) for certain $f,g\in R[\bar{G};\sigma,\tau]$
(see for example \cite[Theorem~10.28]{Lam2}).
\end{proof}

Following \cite{Dicks&Lewin}, let $I$ be a set. Suppose that there
is a map $I\rightarrow R((\bar{G};\sigma,\tau,<)),$ given by
$i\mapsto f_i=\sum_{x\in G}\bar x a_{ix},$ such that the following
two conditions hold:
\begin{enumerate}
\item[(1)] $\operatornamewithlimits{\cup}\limits_{i\in
I}\supp(f_i)$ is well ordered, \item[(2)] for each $x\in G$ the set
$\{i\in I\mid x\in\supp(f_i)\}$ is finite.
\end{enumerate}
Then we say that $\sum_{i\in I}f_i$ is \emph{defined in
$R((\bar{G};\sigma,\tau,<))$}. In this situation $\sum_{i\in I}f_i$
will be used to denote $\sum_{x\in G}\bar
x\Big(\sum_{\scriptscriptstyle \{i\mid x\in \supp f_i\}}
a_{ix}\Big).$ Note that $\sum_{i\in I}f_i$ is then an element of
$R((\bar{G};\sigma,\tau,<)).$

Let $I,J$ be sets, and let $\sum\limits_{i\in I}f_i$,
$\sum\limits_{i\in I}g_i$, $\sum\limits_{j\in J}h_j$ be defined in
$R((\bar{G};\sigma,\tau,<))$. As can be seen in \cite{Dicks&Lewin},
the following hold true:
\begin{enumerate}[\rm (i)]
\item For any $a\in R,$ $\sum\limits_{i\in I} f_ia$ is
defined in $R((\bar{G};\sigma,\tau,<))$ and equals $\Big(
\sum\limits_{i\in I}f_i\Big)a.$
\item $\sum\limits_{i\in I}(f_i+g_i)$ is defined in $R((\bar{G};\sigma,\tau,<))$ and
equals $\sum\limits_{i\in I}f_i + \sum\limits_{i\in I}g_i.$ \item
$\sum\limits_{(i,j)\in I\times J}(f_ih_j)$ is defined in
$R((\bar{G};\sigma,\tau,<))$ and equals $(\sum\limits_{i\in
I}f_i)(\sum\limits_{j\in J}h_j).$
\end{enumerate}

Now we use this to prove that Malcev-Neumann series ring behave very
much like crossed products.

\begin{prop}\label{prop:MNfieldsascrossedproducts}
Let $K$ be a skew field and $(G,<)$ an ordered group. Consider a
crossed product $K[\bar{G};\sigma,\tau]$ and its Malcev-Neumann
series ring $K((\bar{G};\sigma,\tau,<))$. Let $N$ be a normal
subgroup of $G$. The following statements hold true.
\begin{enumerate}[\rm (1)]
\item The crossed product  structure of $K[\bar{G};\sigma,\tau]=K[\bar{N};\sigma,\tau][\overline{G/N};\tilde{\sigma},\tilde{\tau}]$
extends to
$K(\bar{N};\sigma,\tau)[\overline{G/N};\tilde{\sigma},\tilde{\tau}]$,
and to
$K((\bar{N};\sigma,\tau,<))[\overline{G/N};\tilde{\sigma},\tilde{\tau}]$
in the natural way. Therefore
\begin{eqnarray*}
K[\bar{N};\sigma,\tau][\overline{G/N};\tilde{\sigma},\tilde{\tau}] &
\hookrightarrow &
K(\bar{N};\sigma,\tau)[\overline{G/N};\tilde{\sigma},\tilde{\tau}]
\\ & \hookrightarrow &
K((\bar{N};\sigma,\tau,<))[\overline{G/N};\tilde{\sigma},\tilde{\tau}]
\\ & \hookrightarrow  & K((\bar{N};\sigma,\tau,<))((\overline{G/N};\tilde{\sigma},\tilde{\tau},\prec)).
\end{eqnarray*}
\item If the group $G/N$ is Ore embeddable, then $K(\bar{G};\sigma,\tau)$ is the Ore field of
fractions of
$K(\bar{N};\sigma,\tau)[\overline{G/N};\tilde{\sigma},\tilde{\tau}]$.

\item If $N$ is a convex normal subgroup of $G$, then moreover
\begin{enumerate}[\rm (i)]
\item $K((\bar{G};\sigma,\tau,<))=K((\bar{N};\sigma,\tau,<))((G/N;\tilde{\sigma},\tilde{\tau},\prec))$.

\item $K(\bar{G};\sigma,\tau)=K(\bar{N};\sigma,\tau)(\overline{G/N};\tilde{\sigma},\tilde{\tau})$, that is,
$K(\bar{G};\sigma,\tau)$ is the Malcev-Neumann skew field of
fractions of
$K(\bar{N};\sigma,\tau)[\overline{G/N};\tilde{\sigma},\tilde{\tau}]$.
\end{enumerate}
\end{enumerate}
\end{prop}

\begin{proof}
Fix a transversal $\{x_\alpha\mid \alpha\in G/N\}$ of $N$ in $G$
with $x_1=1$. Set $\bar{\alpha}=x_\alpha$ for each $\alpha\in G/N$.
Set $S=K[\bar{G};\sigma,\tau]$. Consider the structure of
$K[\bar{G};\sigma,\tau]$ as a crossed product
$S[G/N;\tilde{\sigma},\tilde{\tau}]$ as in
Section~\ref{sec:crossedproducts}.

\smallskip

Given $\alpha\in G/N$, $\tilde{\sigma}(\alpha)\colon S\rightarrow S$
is defined by $y\mapsto \bar{x}_\alpha^{-1} y\bar{x}_\alpha$. It can
be clearly extended to an automorphism of
$K((\bar{N};\sigma,\tau,<))$. We now prove that
$\tilde{\sigma}(\alpha)$ can also be extended to an automorphism of
$K(\bar{N};\sigma,\tau)$. First note that
$K(\bar{N};\sigma,\tau)=\cup_{i\geq 0}S_i$ where $S_0=S$, and for
$i\geq 0$, $S_{i+1}$ is the subring generated by $S_i$ and the
inverses of the nonzero elements of $S_i$. Of course
$\tilde{\sigma}(\alpha)$ gives an isomorphism of $S_0$. Suppose that
$i$ is such that $\tilde\sigma(\alpha)$ gives an isomorphism of
$S_i$. Now notice that if $f\in S_i\setminus\{0\}$,
$\bar{x}_\alpha^{-1} f^{-1}\bar{x}_\alpha=(\bar{x}_\alpha^{-1}
f\bar{x}_\alpha)^{-1}\in S_{i+1}$. Hence $\tilde{\sigma}(\alpha)
(S_{i+1})\subseteq S_{i+1}$, and clearly it is injective. Also
$f^{-1}=\bar{x}_\alpha^{-1} (\bar{x}_{\alpha}
f^{-1}\bar{x}_\alpha^{-1}) \bar{x}_\alpha=\bar{x}_\alpha^{-1}
(\tilde{\sigma}(\alpha)^{-1}(f))^{-1} \bar{x}_\alpha\in S_{i+1}$.
Therefore, $\tilde{\sigma}(\alpha)$ can be seen as an isomorphism of
$K(\bar{N};\sigma,\tau)$. We will denote these extensions of
$\tilde{\sigma}(\alpha)$ again as $\tilde{\sigma}(\alpha)$

\smallskip

Note that the elements of the set $\{\bar{\alpha}\mid \alpha\in
G/N\}$ are  linearly independent over $K((\bar{N};\sigma,\tau,<))$,
and hence over $K(\bar{N};\sigma,\tau)$.

Let $A=\sum_{\alpha\in G/N}f_\alpha \bar{\alpha}$ and
$B=\sum_{\alpha\in G/N}g_\alpha\bar{\alpha}$ be defined in
$K((\bar{G};\sigma,\tau,<))$ where $f_\alpha,g_\alpha\in
K((\bar{N};\sigma,\tau,<))$ (or $f_\alpha,g_\alpha\in
K(\bar{N};\sigma,\tau)$) for all $\alpha\in G/N$. Then, by the
discussion previous to the statement of
Proposition~\ref{prop:MNfieldsascrossedproducts},
\begin{eqnarray}\label{eq:seriesoperations}
A+B=\sum_{\alpha\in G/N}\bar{\alpha}(f_\alpha+g_\alpha),\quad A\cdot
B=\sum_{\alpha\in
G/N}\bar{\alpha}\Big(\sum_{\beta\gamma=\alpha}\tilde{\tau}(\beta,\gamma)
f_\beta^{\tilde{\sigma}(\gamma)} g_\gamma\Big).
\end{eqnarray}

Observe that if the sets $\supp_{G/N} A=\{\alpha\mid f_\alpha\neq
0\}$ and $\supp_{G/N} B=\{\alpha\mid g_\alpha\neq 0\}$ are finite,
then $A$ and $B$ are always defined in $K((\bar{G};\sigma,\tau,<))$.
Therefore (1) is proved.

To prove (2), note that $K(\bar{G};\sigma,\tau)$ is the skew
subfield of $K((\bar{G};\sigma,\tau,<))$ generated by
$S[G/N;\tilde{\sigma},\tilde{\tau}]$. The skew field
$K(\bar{G};\sigma,\tau)$ clearly contains $K(\bar{N};\sigma,\tau)$
and the set $\{\bar{\alpha}\mid \alpha\in G/N\}$, thus it contains
$K(\bar{N};\sigma,\tau)[G/N;\tilde{\sigma},\tilde{\tau}]$, and hence
$K(\bar{G};\sigma,\tau)$ contains the Ore field of fractions of
$K(\bar{N};\sigma,\tau)[\overline{G/N};\tilde{\sigma},\tilde{\tau}]$.
On the other hand,
$K(\bar{N};\sigma,\tau)[\overline{G/N};\tilde{\sigma},\tilde{\tau}]$
contains
$S[\overline{G/N};\tilde{\sigma},\tilde{\tau}]=K[\bar{G};\sigma,\tau]$.

\medskip
Now we prove (3)(i). For each $x\in G$ there exist unique $\alpha\in
G/N$ and $n_\alpha\in N$ such that $x=x_\alpha n$.

Let $f=\sum_{x\in G}\bar{x}a_x\in K((\bar{G};\sigma,\tau,<))$ where
$a_x\in K$ for each $x\in G$. For each $x\in G$,
$\bar{x}a_x=\bar{\alpha}\bar{n}\zeta_{n\alpha}$ for some
$\zeta_{n\alpha}\in K$. For each $\alpha\in G/N$, let
$f_\alpha=\sum_{n\in N} \bar{n}\zeta_{n\alpha}$. Then $f_\alpha\in
K((\bar{N};\sigma,\tau,<))$. Otherwise if
$n_1>n_2>\dotsb>n_r>\dotsb$ is a strictly descending chain of
elements in $\supp f_\alpha$, then $x_\alpha n_i\in\supp f$, and
$x_\alpha n_1>x_\alpha n_2>\dotsb>x_\alpha n_r>\dotsb$, a
contradiction. Also, the set $\{\alpha\in G/N\mid f_\alpha\neq 0\}$
is well ordered. Otherwise, if
$\alpha_1>\alpha_2>\dotsb>\alpha_r>\dotsb$, there is
$x_i=x_{\alpha_i}n_i\in\supp f$. Then, by the convexity of $N$,
$g_1>g_2>\dotsb>g_r>\dotsb$, a contradiction.

Hence $\sum_{\alpha\in G/N} \bar{\alpha}f_\alpha$ is defined in
$K((\bar{G};\sigma,\tau,<))$ and equals $f$.

On the other hand, given a well ordered set $\Delta\subseteq G/N$,
and series  $f_\alpha=\sum_{n\in N}\bar{n}\zeta_{n\alpha}\in
K((\bar{N};\sigma,\tau,<))$ for each $\alpha\in\Delta$. Then it is
not very difficult to see that the series
$\sum_{\alpha\in\Delta}\bar{\alpha}f_\alpha$ is defined in
$K((\bar{G};\sigma,\tau,<))$. Therefore
$K((\bar{G};\sigma,\tau,<))=K((\bar{N};\sigma,\tau,<))((\overline{G/N};\tilde{\sigma},\tilde{\tau},\prec))$
as sets. But \eqref{eq:seriesoperations} shows that they are equal
as rings.

To show (3)(ii) proceed as in the proof of (2).
\end{proof}

We remark that
Proposition~\ref{prop:MNfieldsascrossedproducts}(3)(ii) is an
important particular case of \cite{Hughes2}, but the proof here is
much easier.


\section{Main result}\label{sec:proofofmainresult}

The next theorem is the main result of the paper. This section is
devoted to prove it.

\begin{theo}\label{theo:mainresult}
Let $(G,<)$ be a noncommutative ordered group. Let $K$ be a skew
field, and let $K[G]$ be the group ring over $K$.  Then $K(G)$
contains noncommutative free group algebras over its center.
\end{theo}

The proof of Theorem~\ref{theo:mainresult} is divided in three
parts, depending on which kind of ordered group $(G,<)$ is.  We will
single out three types of ordered groups.
\begin{enumerate}[Type 1:]
\item Every convex jump in $(G,<)$ is central, that is, $[H,G]\subseteq N$ for
all convex jumps $(N,H)$ of $(G,<)$.
\item Every convex subgroup of $(G,<)$ is normal in $G$, but there exists a
convex jump $(N,H)$ which is not central, that is, $[H,G]\nsubseteq
N$.
\item There exists a convex subgroup of $(G,<)$ which is not normal in $G$.
\end{enumerate}

We note that any ordered group $(G,<)$ falls into one (and only one)
of these three types of ordered groups. This is because if $(G,<)$
is of type~1, then every convex subgroup is normal in $G$ as noted
in \cite[p.~226]{BludovGlassRhemtulla1}. Indeed, let $U$ be a convex
subgroup of $(G,<)$, $u\in U$ and $x\in G$. Consider the convex jump
$(N_u,H_u)$ determined by $u$. Then $N_u\subset H_u\subseteq U$, and
$x^{-1}ux=u[u,x]\in U$ since $[H_u,G]\subseteq N_u$ by hypothesis.

On the other hand it may happen  that there exist two total orders
$<_1,\ <_2$ of an orderable group $G$ such that $(G,<_1)$ and
$(G,<_2)$ are ordered groups of different types.

\subsection{Ordered groups of
Type~1}\label{section:oderedgroupstype1}

The following are sufficient conditions to extend morphisms of rings
 $\varphi\colon R_1\rightarrow R_2$   and morphisms of groups $\eta\colon
G_1\rightarrow G_2$  to morphisms of rings $\Phi\colon
R_1[\bar{G}_1;\sigma_1,\tau_1]\rightarrow
R_2[\bar{G}_2;\sigma_2,\tau_2]$ and $\Phi\colon
R_1((\bar{G}_1;\sigma_1,\tau_1,<))\rightarrow
R_2((\bar{G}_2;\sigma_2,\tau_2,<))$.

\begin{lem}\label{prop:conditionforextensionmorphism}
Let $\varphi\colon R_1\rightarrow R_2$ be a morphism of rings, and
$\eta\colon G_1\rightarrow G_2$ a morphism of groups. Consider
crossed product group rings $R_1[\bar{G}_1;\sigma_1, \tau_1]$ and
$R_2[\bar{G}_2;\sigma_2,\tau_2]$. Suppose that
\begin{eqnarray}
\varphi(r^{\sigma_1(x)}) & = &\varphi(r)^{\sigma_2(\eta(x))},\label{eq:conditionforextensionmorphism1}\\
\varphi(\tau_1(x,y)) & =  & \tau_2(\eta(x),\eta(y)).
\label{eq:conditionforextensionmorphism2}
\end{eqnarray}
Then the following hold true
\begin{enumerate}[\rm(1)]
\item The map $\Phi\colon R_1[\bar{G}_1;\sigma_1, \tau_1]\rightarrow R_2[\bar{G}_2;\sigma_2,\tau_2]$, defined by
\begin{equation}
\Phi\Big(\sum_{x\in G_1}\bar{x}a_x\Big)=\sum_{x\in
G_1}\overline{\eta(x)}\varphi(a_x),
\end{equation}
is a morphism of rings.
\item If, moreover, $(G_1,<_1)$, $(G_2,<_2)$ are ordered groups and $\eta\colon
G_1\rightarrow G_2$ is an injective morphism of ordered groups. Then
$\Phi\colon R_1((\bar{G}_1;\sigma_1,\tau_1,<_1))\rightarrow
R_2((\bar{G}_2;\sigma_2,\tau_2,<_2))$, the natural extension of
$\Phi$ in {\rm (1)}, defined  by
\begin{equation} \Phi\Big(\sum_{x\in G_1}\bar{x}a_x\Big)=\sum_{x\in
G_1}\overline{\eta(x)}\varphi(a_x),
\end{equation}
is a morphism of rings.
\end{enumerate}
\end{lem}

\begin{proof}
The map $\Phi$ in (2) is well defined because $\eta$ is an injective
morphism of ordered groups. Now the same proof works for (1) and
(2). By definition,
$\Phi(1)=\Phi(\bar{1})=\overline{\eta(1)}=\bar{1}=1$. Let
$\sum_{x\in G_1}\bar{x}a_x,\ \sum_{x\in G_1}\bar{x}b_x\in
R_1((\bar{G}_1;\sigma_1,\tau_1,<_1))$.  Then
\begin{eqnarray*}
\Phi\Big(\Big(\sum_{x\in G_1}\bar{x}a_x\Big)\Big(\sum_{x\in
G_1}\bar{x}b_x\Big)\Big) &= & \Phi\Big(\sum_{x\in
G_1}\bar{x}\Big(\sum_{yz=x}\tau_1(y,z)a_y^{\sigma_1(z)}b_z\Big)\Big) \\
&=& \sum_{x\in
G_1}\overline{\eta(x)}\Big(\sum_{yz=x}\varphi\big(\tau_1(y,z)\big)\varphi(a_y^{\sigma_1(z)})\varphi(b_z)\Big)\\
&=& \sum_{x\in
G_1}\overline{\eta(x)}\Big(\sum_{yz=x}{\tau_2}\big(\eta(y),\eta(z)\big)\varphi(a_y)^{\sigma_2(\eta(z))}\varphi(b_z)\Big)\\
&=& \Big(\sum_{x\in
G_1}\overline{\eta(x)}\varphi(a_x)\Big)\Big(\sum_{x\in
G_1}\overline{\eta(x)}\varphi(b_x)\Big)\\
&=& \Phi\Big(\sum_{x\in G_1}\bar{x}a_x\Big)\Phi\Big(\sum_{x\in
G_1}\bar{x}b_x\Big).
\end{eqnarray*}
Analogously, it can be seen that $\Phi$ is additive.
\end{proof}

The trivial situation where the crossed products
$R_1[\bar{G}_1;\sigma_1,\tau_1]$ and
$R_2[\bar{G}_2;\sigma_2,\tau_2]$ are group rings is well known: for
any morphism of rings $\varphi\colon R_1\rightarrow R_2$ and any
morphism of groups $\eta\colon G_1\rightarrow G_2$ conditions
\eqref{eq:conditionforextensionmorphism1} and
\eqref{eq:conditionforextensionmorphism2} are satisfied because
$\sigma_1$, $\sigma_2$ and $\tau_1$, $\tau_2$ are trivial.

\medskip

Although, as pointed out in \cite[p.~674]{Lichtmanuniversalfields},
the proof of
\cite[Proposition~3.1]{LichtmanOnembeddingofgroupringsofresidually}
is not correct, some minor changes give
Example~\ref{ex:exampleofextension}. It sets the situation where we
will use Lemma~\ref{prop:conditionforextensionmorphism} and it also
fixes the notation that  will be used in what follows.

\begin{ex}\label{ex:exampleofextension}
Let $G$ be a group with a normal subgroup $N$. Let $K$ be a skew
field. As in Section~\ref{sec:crossedproducts}, we look at the group
ring $K[G]$ as a crossed product
$S[G/N;\tilde{\sigma},\tilde{\tau}]$ where $S=K[N]$. Fix a
transversal $\{x_\alpha\mid \alpha\in G/N\}$ of $N$ in $G$ with
$x_1=1$. In our situation  the maps $\tilde{\sigma}$ and
$\tilde{\tau}$ are defined as follows. For each $\alpha\in G/N$,
$\tilde{\sigma}(\alpha)$ is given by $x_\alpha^{-1} r x_{\alpha}$
for all $r\in S$; for each $\alpha,\beta\in G/N$,
$\tilde{\tau}(\alpha,\beta)=n_{\alpha\beta}$ where $n_{\alpha\beta}$
is the unique element in $N$ such that
$x_{\alpha\beta}n_{\alpha\beta}=x_\alpha x_\beta$.

Let $R_1=S$, $R_2=K$, $G_1=G_2=G/N$, $\eta\colon G/N\rightarrow G/N$
be the identity, and $\varphi\colon S\rightarrow K$  the
augmentation map. Then there exists a morphism $\Phi\colon
K[G]=S[\overline{G/N};\tilde{\sigma},\tilde{\tau}]\rightarrow
K[G/N]$ extending $\varphi$ and $\eta$. Indeed,
$$\varphi(\tilde{\tau}(\alpha,\beta))=\varphi(n_{\alpha\beta})=1=\tau_2(\eta\alpha,\eta\beta),$$
$$\varphi(r^{\tilde{\sigma}(\alpha)})=\varphi(x_\alpha r x_\alpha^{-1})=\varphi (r)=\varphi(r)^{\sigma_2(\alpha)},$$
for all $\alpha,\beta\in G/N$ and $r\in S$. Thus  conditions
\eqref{eq:conditionforextensionmorphism1} and
\eqref{eq:conditionforextensionmorphism2} are satisfied.

Suppose moreover that $N$ is such that $G/N$ is orderable. If
$(G/N,\prec)$ is an ordered group, then $\Phi$ can be extended to
$\Phi\colon S((G/N;\tilde{\sigma},\tilde{\tau},\prec))\rightarrow
K((G/N;\prec))$ by
Lemma~\ref{prop:conditionforextensionmorphism}(2).

Given a series $A=\sum_{\alpha\in G/N} \alpha a_\alpha\in
K((G/N;\prec))$, we define the \emph{good preimage} $\hat{A}$ of $A$
by $\Phi$ as  $$\hat{A}=\sum_{\alpha\in G/N} \bar{\alpha}a_\alpha\in
S((\overline{G/N};\tilde{\sigma},\tilde{\tau},\prec)).$$ Note that
if $A$ is invertible in $K((G/N;\prec))$, then $\hat{A}$ is
invertible in
$S((\overline{G/N};\tilde{\sigma},\tilde{\tau},\prec))$ because
$a_{\alpha_0}\in K$ is invertible in $S$, where $\alpha_0=\min\supp
A=\min\supp\hat{A}$.

Let $P$ be a central subfield of $K$. Then  it is a central subfield
of $K((G/N;\prec))$ and
$S((\overline{G/N};\tilde{\sigma},\tilde{\tau},\prec))$. Let $I$ be
a set, and $A_i,B_i\in K((G/N;\prec))\setminus\{0\}$ for each $i\in
I$. Suppose that $\{A_iB_i^{-1}\}_{i\in I}$ generate a free group
algebra over $P$ inside $K((G/N;\prec))$. Let
$\{\hat{A}_i,\hat{B}_i\}_{i\in I}$ be the set of good preimages of
$\{A_i,B_i\}_{i\in I}$. For each $i\in I$, $\hat{A}_i\hat{B}_i^{-1}$
is an invertible element of
$S((\overline{G/N};\tilde{\sigma},\tilde{\tau},\prec))$. Then
$\{\hat{A}_i\hat{B}_i^{-1}\}_{i\in I}$ generate a free group algebra
over $P$ inside
$S((\overline{G/N};\tilde{\sigma},\tilde{\tau},\prec))$ because of
the following reasons: $\Phi$ is a morphism of rings such that
$\Phi(\hat{A}_i\hat{B}_i^{-1})=A_iB_i^{-1}$  for all $i\in I$, the
series $\hat{A}_i\hat{B}_i^{-1}$ is invertible in
$S((\overline{G/N};\tilde{\sigma},\tilde{\tau},\prec))$ for all
$i\in I$, $\Phi(a)=a$ for all $a\in P$, and the elements of $P$
commute with the elements of $\{\hat{A}_i\hat{B}_i^{-1}\}_{i\in I}$.
\qed
\end{ex}

\begin{prop}\label{prop:reductiontonilpotent}
Let $(G,<)$ be an ordered group with a normal  subgroup $N$ such
that $G/N$ is a noncommutative torsion-free nilpotent group. Let $K$
be a skew field with prime subfield $P$. Consider the group ring
$K[G]$ and the Malcev-Neumann series ring $K((G;<))$. Then $K(G)$
contains a noncommutative free group $P$-algebra.
\end{prop}

\begin{proof}
Recall that if $H$ is a nilpotent group and  $L$ a skew field then
any crossed product $L[\bar{H};\sigma,\tau]$ is an Ore domain. We
denote its Ore skew field of fractions by
$Q_{cl}(L[\bar{H};\sigma,\tau])$. Thus
$Q_{cl}(L[\bar{H};\sigma,\tau])$ embeds in any skew field that
contains $L[\bar{H};\sigma,\tau]$.

Set $S=K[N]$. As in Example~\ref{ex:exampleofextension},
$K[G]=S[\overline{G/N};\tilde{\sigma},\tilde{\tau}]$.   By
Proposition~\ref{prop:MNfieldsascrossedproducts}(1),
$S[\overline{G/N};\tilde{\sigma},\tilde{\tau}]\hookrightarrow
K(N)[\overline{G/N};\tilde{\sigma},\tilde{\tau}]$. Moreover, $K(G)$
is the Ore field of fractions of
$K(N)[\overline{G/N};\tilde{\sigma},\tilde{\tau}]$ by
Proposition~\ref{prop:MNfieldsascrossedproducts}(2).

Fix an order $\prec$ on $G/N$ such that $(G/N, \prec)$ is an ordered
group. The foregoing embedding can be extended to
$S((\overline{G/N};\tilde{\sigma},\tilde{\tau},\prec))\hookrightarrow
K(N)((\overline{G/N};\tilde{\sigma},\tilde{\tau},\prec)).$ Hence we
obtain the following commutative diagram of  ring embeddings
\begin{eqnarray}\label{eq:diagramnotOre}
\xymatrix{S[\overline{G/N};\tilde{\sigma},\tilde{\tau}]\ar@{^{(}->}[rr]\ar@{^{(}->}[dd]
\ar@{^{(}->}[rd] & &
K(N)[\overline{G/N};\tilde{\sigma},\tilde{\tau}] \ar@{^{(}->}[dd]
\ar@{_{(}->}[ld]\\ &  K(G) \ar@{^{(}->}[dr] &\\
S((\overline{G/N};\tilde{\sigma},\tilde{\tau},\prec))\ar@{^{(}->}[rr]
& & K(N)((\overline{G/N};\tilde{\sigma},\tilde{\tau},\prec))& }
\end{eqnarray}

By \cite[Corollary~2.1]{FigueiredoGoncalvesShirvani},
$Q_{cl}(P[G/N])$ contains a noncommutative free group $P$-algebra.
Since $P[G/N]$ is a subring of $K[G/N]$,  $Q_{cl}(P[G/N])\subseteq
Q_{cl}(K[G/N])$. Hence there exist $A_1,B_1, A_2,B_2\in K[G/N]$ such
that $A_1B_1^{-1}, A_2B_2^{-1}$ generate a noncommutative free group
$P$-algebra  inside $Q_{cl}(K[G/N])\subseteq K((G/N;\prec))$.
Consider the morphism of rings $\Phi\colon
S((\overline{G/N};\tilde{\sigma},\tilde{\tau},\prec))\rightarrow
K((G/N;\prec))$ constructed in Example~\ref{ex:exampleofextension}.
Let $\hat{A}_i,\hat{B}_i\in
S((\overline{G/N};\tilde{\sigma},\tilde{\tau},\prec))$, $i=1,2$, be
good preimages of $A_i,B_i$ as in
Example~\ref{ex:exampleofextension}.
 Hence, by
Example~\ref{ex:exampleofextension}, they generate a noncommutative
free group $P$-algebra  inside
$S((\overline{G/N};\tilde{\sigma},\tilde{\tau},\prec))$. Observe
that $\hat{A}_i,\hat{B}_i\in K[G]$, and therefore, this free group
algebra is also inside $K(G)$ by diagram \eqref{eq:diagramnotOre}.
\end{proof}

\noindent {\bf Proof of Theorem~\ref{theo:mainresult} for ordered groups of Type~1}.
Suppose that $(G,<)$ is a noncommutative ordered group of Type~1.

Let $x,y$ be positive elements of $(G,<)$ such that $[x,y]\neq 1$.
If $z=\min\{x,y\}$ and $(N_z,H_z)$ is the convex jump determined by
$z$, then $[x,y]\in N_z$ since every convex jump is central.
Moreover, neither $x$ nor $y$ belong to $N_z$.

Let $(N,H)$ be the convex jump determined by $[x,y]$. Hence
$[x,y]\in H\setminus N$. Moreover, since $H\subseteq N_z$, neither
$x$ nor $y$ belong to $H$.

Let $A$ be the subgroup of $G$ generated by $\{x,y\}$, and let
$B=N\cap A$. First observe that $A/B\hookrightarrow G/N$ which
proves that $A/B$ is torsion free because $G/N$ is an orderable
group. Since $[x,y]\notin N$, $A/B$ is not commutative. On the other
hand $A/(H\cap A)$ is a nontrivial commutative group because
$[x,y]\in H$. Hence $[A,A]\subseteq H\cap A$. Therefore, since all
convex jumps are central, $[[A,A],A]\subseteq N\cap A=B$. Therefore
$A/B$ is a torsion-free nilpotent group of index two. By
Proposition~\ref{prop:reductiontonilpotent},  there exist
noncommutative free group algebras over the prime subfield $P$ of
$K$ inside $K(A)\subseteq K(G)$. Now
Lemma~\ref{lem:freegroupalgebraovercentre} implies the result. \qed

\subsection{Ordered groups of
Type~2}\label{sec:orderedgroupsoftype2}

The strategy to prove Theorem~\ref{theo:mainresult} in this case is
to find  a noncommutative free monoid inside ordered groups of
type~2. The following key result provides a useful criterion for
finding free monoids inside groups. It is
\cite[Corollary~2.5]{Rosenblattinvariantmeasures}.

\begin{lem}\label{lem:monoidpingponglemma}
A group $G$ contains a free monoid on two generators if and only if
there exist a nonempty subset $A\subseteq G$ and elements $a,b\in G$
such that $aA\cup bA\subseteq A$ and $aA\cap bA=\emptyset$. In this
case $a,b$ generate a free monoid. \qed
\end{lem}

The following is a slight modification of
\cite[Lemma~4.16]{Rosenblattinvariantmeasures}.

\begin{lem}\label{lem:sumofrealnumbers}
Let $r\in\mathbb{R}$ with $r\geq 2$ or $0<r\leq \frac{1}{2}$. Let
there be integers $0\leq a_1<a_2<\dotsb<a_n$ and $0\leq
b_1<b_2<\dotsb < b_m$. If
\begin{equation}\label{eq:sumsofrealnumbers}
r^{a_1}+r^{a_2}+\dotsb +r^{a_n}=r^{b_1}+r^{b_2}+\dotsb r^{b_m},
\end{equation}
then $n=m$ and $(a_1,\dotsc,a_n)=(b_1,\dotsc,b_m)$.
\end{lem}

\begin{proof}
First suppose that $r\geq 2$. We proceed by induction on the maximum
number of summands in \eqref{eq:sumsofrealnumbers}. If $1=n\geq m$,
the result is clear. If $a_n\neq b_m$, suppose that $a_n>b_m$ (the
other case is proved in the same way). Then $r^{a_n}=\sum_{i=1}^m
r^{b_i}-\sum_{i=1}^{n-1}r^{a_i}\leq \sum_{i=1}^m r^{b_i}\leq
\sum_{l=0}^{b_m}r^l=\frac{r^{b_m}-1}{r-1}\leq r^{b_m}-1\leq
r^{a_n}-1,$ a contradiction. Thus $a_n=b_m$. Now we can apply
induction to obtain the result.

Suppose now that $0<\frac{1}{2}\leq r$. If $a_n\geq b_m$, we
multiply the expression by $r^{-a_n}$ to obtain
$(r^{-1})^0+(r^{-1})^{(a_n-a_{n-1})}+\dotsb+(r^{-1})^{(a_n-a_1)}=(r^{-1})^{(a_n-b_m)}+\dotsb+(r^{-1})^{(a_n-b_1)}$
and apply the case $2\leq r^{-1}$.
\end{proof}

The next lemma is basically proved in
\cite[Examples~4.19]{Rosenblattinvariantmeasures}.

\begin{lem}\label{lem:freemonoid2ndordering}
Let $(H,+)$ be a nontrivial additive sugroup of $(\mathbb{R},+)$ and
$C=\langle x\rangle$ an infinite cyclic group. Consider the
semidirect product $G=H\rtimes C$ where $x$ acts on $H$ by
multiplication by $r\in \mathbb{R}$, where either $0<r\leq
\frac{1}{2}$ or $r\geq 2$. That is, $xzx^{-1}=rz$ for all $z\in H$.
Let $t\in H\setminus\{0\}$, then $\{tx,x\}$ generate a free
noncommutative submonoid of $G$.
\end{lem}

\begin{proof}
Let $A$ be the subset of $G$ defined by $$A=\{(r^{a_1}+\dotsb+
r^{a_n})t\cdot x^j\mid 0\leq a_1<a_2<\dotsb<a_n,\ a_i\in\integers,\
n,j\in\naturals \}.$$ Observe that $tx\cdot(r^{a_1}+\dotsb
+r^{a_n})t\cdot x^j=(t+x(r^{a_1}+\dotsb +r^{a_n})x^{-1})\cdot
x^{j+1}=(1+r^{a_1+1}+\dotsb +r^{a_n+1})t\cdot x^{j+1}$.  On the
other hand $x\cdot (r^{b_1}+\dotsb +r^{b_m})t\cdot
x^l=x(r^{b_1}+\dotsb +r^{b_m})tx^{-1}\cdot
x^{l+1}=(r^{b_1+1}+\dotsb+ r^{b_m+1})t\cdot x^{l+1}$. Thus
$txA\subseteq A$ and $xA\subseteq A$. But by
Lemma~\ref{lem:sumofrealnumbers}, $txA\cap xA=\emptyset$. Therefore,
$\{tx,x\}$ generate a noncommutative free monoid by
Lemma~\ref{lem:monoidpingponglemma}.
\end{proof}

\noindent {\bf Proof of Theorem~\ref{theo:mainresult} for ordered
groups of Type~2.} Let $(G,<)$ be an ordered group of Type~2. Let
$(B,L)$ be a convex jump  of $(G,<)$ such that $[L,G]\nsubseteq B$.

We want to find a noncommutative free monoid  generated by two
positive elements of $(G,<)$. For this it is enough to find two
positive elements of $(G/B,\prec)$ that generate a noncommutative
free submonoid of $G/B$.

Set $H=L/B$ which is an additive subgroup of $(\mathbb{R},+)$. Since
$[L,G]\nsubseteq B$ and $L\lhd G$, there exists $x\in G/B$ such that
conjugation by $x$ defines a nonidentity morphism of the ordered
group $H$. Hence $x$ acts on $H$ as multiplication by a positive
real number $r\neq 1$. Changing to $x^n$ if necessary, we can
suppose that either $r\geq 2$ or $0<r\leq \frac{1}{2}$. Clearly, if
$C$ is the infinite cyclic group generated by $x$, we obtain that
the subgroup generated by $H$ and $C$ is a semidirect product
$H\rtimes C$. Then by Lemma~\ref{lem:freemonoid2ndordering} we
obtain the free monoid of positive elemens inside $G/B$ as desired.

Now apply Lemma~\ref{lem:freemonoidfreegroupalgebra} to obtain a
noncommutative free group algebra inside $K(G)$. \qed

\subsection{Ordered groups of
Type~3}\label{sec:orderedgroupsoftype3} The proof of
Theorem~\ref{theo:mainresult} for ordered groups of type~3 follows
directly from the following result and
Lemma~\ref{lem:freemonoidfreegroupalgebra}.
Lemma~\ref{lem:convexnotnormal} was first proved in
\cite[Lemma~8]{LongobardiMajRhemtullanofree}. Our proof is  based on
the ordering on the group and exhibits the generators of the free
monoid when the order is well known.

\begin{lem}\label{lem:convexnotnormal}
Let $(G,<)$ be an ordered group  with a convex subgroup $C$ which is
not normal in $G$, then $G$ contains a noncommutative free monoid.
\end{lem}

\begin{proof}
First recall that for every $x\in G$, $xCx^{-1}$ is a convex
subgroup of $(G,<)$. Thus, since the convex subgroups of $(G,<)$
form a chain, there exists $b\in G$ such that $bCb^{-1}\varsubsetneq
C$. Indeed, if $C\varsubsetneq xCx^{-1}$, then
$x^{-1}Cx\varsubsetneq C$. Notice that $b\notin C$, and if
$(N_b,H_b)$ is the convex jump determined by $b$, $C\subseteq
N_b\varsubsetneq H_b$. Choose $a\in C$, such that $a\notin bCb^{-1}$
and $a>1$ if $b>1$ or $a<1$  in case $b<1$. Observe that for each
$n\geq 0$, $b^{n+1}Cb^{-(n+1)}\varsubsetneq b^nCb^{-n}$ and
$b^nab^{-n}\in b^nCb^{-n}\setminus b^{n+1}Cb^{-(n+1)}$. Let
$(N_0,H_0)$ be the convex jump determined by $a$. Then $H_0\subseteq
C\varsubsetneq H_b$. For each $n\geq 1$, let $(N_n,H_n)$ be the
convex jump determined by $b^nab^{-n}$. Then $b^nab^{-n}\in H_n$ but
$b^nab^{-n}\notin N_n$. Hence, by definition of $H_{n+1}$ and $N_n$,
$H_{n+1}\subseteq b^{n+1}Cb^{-(n+1)}\subseteq N_n$. Therefore
$H_{n+1}\varsubsetneq H_n$ for all $n\geq 0$.

We claim that $\{a,b\}$ generate a noncommutative free monoid.
Suppose that $\{a,b\}$ satisfy a positive word. We can suppose that
\begin{equation}\label{eq:positiveword}
a^{n_1}b^{m_1}\cdots a^{n_s}b^{m_s}=b^{l_1}a^{k_2}b^{l_2}\cdots
a^{k_t}b^{l_t}
\end{equation}
with $m_i,n_i\geq 0$, $l_j,k_j\geq 0$ for all $i$, $j$ and $n_1\geq
1$, $l_1\geq 1$. Then \eqref{eq:positiveword} can be written as
\begin{eqnarray}
& a^{n_1}(b^{m_1}a^{n_2}b^{-m_1})
\cdots (b^{m_1+\dotsb+m_{s-1}}a^{n_s}b^{-(m_1+\dotsb+m_{s-1})})b^{m_1+\dotsb+m_s}\nonumber\\
&= &   \label{eq:equalityofwords} \\
& (b^{l_1} a^{k_2}b^{-l_1}) \cdots (b^{l_1+\dotsb
+l_{t-1}}a^{k_t}b^{-(l_1+\dotsb+l_{t-1})})b^{l_1+\dotsb+l_t}\nonumber
\end{eqnarray}
Set $l=l_1+\dotsb+l_t$ and $m=m_1+\dotsb+m_s$. Suppose that $l<m$.
Then $b^{m-l}\in H_0$ by \eqref{eq:equalityofwords} because
$a^{n_1}\in H_0$,
$b^{(m_1+\dotsb+m_i)}a^{n_{i+1}}b^{-(m_1+\dotsb+m_i)}\in H_0$ for
all $1\leq i\leq s-1$ and
$b^{l_1+\dotsb+l_j}a^{k_{j+1}}b^{-(l_1+\dotsb+l_j)}\in H_0$ for all
$1\leq j\leq t-1$. This is a contradiction because $b^{m-l}\in
H_b\setminus H_0$. Analogously we obtain a contradiction if $m<l$.
Suppose now that $l=m$. Then \eqref{eq:equalityofwords} shows that
$a^{n_1}\in H_1$ because
$b^{(m_1+\dotsb+m_i)}a^{n_{i+1}}b^{-(m_1+\dotsb+m_i)}\in H_1$ for
all $1\leq i\leq s-1$ and
$b^{l_1+\dotsb+l_j}a^{k_{j+1}}b^{-(l_1+\dotsb+l_j)}\in H_1$ for all
$1\leq j\leq t-1$.      This is a contradiction because $a^{n_1}\in
H_0\setminus H_1$.

Notice now that if $a,b<1$ generate a free monoid, then
$1<a^{-1},b^{-1}$ and they also generate a free monoid. Therefore we
can suppose that $a,b$ are positive elements of $(G,<)$ that
generate a noncommutative free monoid.
\end{proof}

\subsection{Remarks}
There are orderable groups $G$ such that for each ordering $<$ on
$G$, $(G,<)$ is of type 2, and moreover, each noncommutative
finitely generated subgroup $H$ of $G$ is also of type 2. For
example each ordering $<$ on $B(1,2)=\{a,b\mid bab^{-1}=a^2\}$ is of
type 2, see for example
\cite[Section~4]{RivasOnspacesofConradorderings}. Then it is easy to
prove that each noncommutative finitely generated subgroup is of
type 2. Hence the argument to find a group algebra when $(G,<)$ is
of type~2 is necessary.

Following \cite{BludovGlassRhemtulla1}, let $\mathcal{C}$ denote the
class of orderable groups in which every order is central i.e. of
type~1, and let $\mathcal{C}_2$ be the class of orderable groups in
which every order on every two-generator subgroup is central. Then
$\mathcal{C}_2$ is not empty because the class of torsion-free
locally nilpotent groups is in $\mathcal{C}_2$
\cite{SmirnovInfrainvariantsubgroups}, and
$\mathcal{C}_2\varsubsetneq \mathcal{C}$
\cite{BludovGlassRhemtulla1}. Hence  the argument to find free group
algebras when $(G,<)$ is of type 1 is necessary for the elements in
$\mathcal{C}_2$. Although every locally soluble group in
$\mathcal{C}_2$ is locally nilpotent \cite{BludovGlassRhemtulla1},
there are finitely generated non-locally-soluble groups in
$\mathcal{C}_2$ \cite{BludovGlassRhemtulla2}. Also there are
non-locally-soluble groups in $\mathcal{C}_2$ which are residually
torsionfree nilpotent without noncommutative free subsemigroups
\cite{BludovGlassRhemtulla2}.

It is known that there are orderable groups $G$ such that every
ordered group $(G,<)$ has no normal convex subgroups (see for
example \cite{ZenkovMedvedevDlabgroups} and the references there).
We do not know whether there exists a group satisfying that every
noncommutative finitely generated subgroup (or every two-generator
subgroup) inherits this property. Hence it could happen that the
argument for ordered groups of type 3 is not necessary because some
subgroup of $G$ could be made into an ordered group of type 1 or 2.

\section{Some consequences}\label{sec:someconsequences}

\begin{coro}\label{coro:maincorollary}
Let $G$ be a noncommutative orderable group. Let $K$ be a skew
field, and let $K[G]$ be the group ring over $K$. If $D$ is any skew
field of fractions of $K[G]$ (it may not be the Malcev-Neumann skew
field of fractions), then $D$ contains a noncommutative free algebra
over its center.

If, moreover, the center of $D$ is uncountable, then $D$ contains a
noncommutative free group algebra over its center.
\end{coro}

\begin{proof}
If $K[G]$ is not an Ore domain, it is well known that $K[G]$
contains a free algebra over the prime subfield $P$ of $K$
\cite[Proposition~10.25]{Lam2}, and therefore $D$ contains a
noncommutative free algebra over the center of $D$ by
Lemma~\ref{lem:freegroupalgebraovercentre}.

If $K[G]$ is a right (or left) Ore domain, then  $D$ is the Ore
field of fractions of $K[G]$. Fix a total order on $G$ such that
$(G,<)$ is an ordered group. Then $K[G]$, and thus $D$, embeds in
$K((G;<))$. Hence $D=K(G;<)$ contains a free group algebra by
Theorem~\ref{theo:mainresult}.

By \cite[Theorem]{GoncalvesShirvani}, if $D$ has uncountable center
$Z$ and contains a noncommutative free algebra over $Z$, then it
contains a noncommutative free group algebra over $Z$.
\end{proof}

Now we show that in fact we have shown something stronger than
Theorem~\ref{theo:mainresult}.

\begin{coro}
Let $(G,<)$ be an orderable group with $G$ not in $\mathcal{C}_2$.
Let $K$ be a skew field and $K[\bar{G};\sigma,\tau]$ a crossed
product group ring. Then $K(\bar{G};\sigma,\tau)$ contains a free
group algebra over its center.
\end{coro}

\begin{proof}
If $(G,<)$ is of type 2 or 3, then we have shown in
Section~\ref{sec:orderedgroupsoftype2} and in
Section~\ref{sec:orderedgroupsoftype3}, respectively, that $G$
contains a noncommutative free monoid.

If $(G,<)$ is of type 1, but $G\notin \mathcal{C}_2$, then there
exists a subgroup $H$ of $G$ and a total order $<'$ of $H$ such that
$(H,<')$ is of type 1 or 2. Hence $H$ contains a noncommutative free
monoid.

Now Lemma~\ref{lem:freemonoidfreegroupalgebra} implies the result.
\end{proof}

We prove the next result with the techniques developed in
Section~\ref{section:oderedgroupstype1}. When $G$ is a residually
torsion-free nilpotent group,
Proposition~\ref{prop:GoncalvesShirvani} gives
\cite[Theorem~5]{GoncalvesMandelShirvaniQuaternions} whose original
proof relied on
\cite[Proposition~3.1]{LichtmanOnembeddingofgroupringsofresidually}.

\begin{prop}\label{prop:GoncalvesShirvani}
Let $K$ be a commutative field of characteristic not 2, and let
$(G,<)$ be an ordered group. Let $x,y$ be two noncommuting elements
of $G$ and $H$ the subgroup of $G$ that they generate. If there
exists a subgroup $N\lhd H$ such that $H/N$ is a noncommutative
torsion-free nilpotent group, then, for any $c,d\in
K\setminus\{0\}$, the subgroup $\langle \{1+cx, 1+dy\}\rangle$ of
 $K(G)$ is a noncommutative
free group.
\end{prop}

\begin{proof}
Since $H\leq G$, $K(H)\subseteq K(G)$. Thus we may suppose that $H=G$ and $N\lhd G$.

We borrow  notation from Example~\ref{ex:exampleofextension} and
Proposition~\ref{prop:reductiontonilpotent}.

Let $\alpha_0$  and $\beta_0$ be the cosets of $x$ and $y$,
respectively, in $G/N$. As in Example~\ref{ex:exampleofextension},
fix a transversal $\{x_\alpha\mid \alpha\in G/N\}$ with $x_1=1$ and
such that $x_{\alpha_0}=x$ and $y_{\beta_0}=y$.

Suppose that $G/N$ is  torsion-free nilpotent of class two. Fix a
total order $\prec$ in $G/N$ such that $(G/N, \prec)$ is an ordered
group, and consider the group ring $K[G/N]$. Then, by
\cite[Proposition~20]{GoncalvesMandelShirvaniQuaternions},
$\{1+c\alpha_0, 1+d\beta_0\}$ generate a free group inside
$Q_{cl}(K[G/N])\subseteq K((G/N,\prec))$. Observe that
$1+c\bar{\alpha}_0=1+cx$ and $1+d\bar{\beta}_0=1+dy$ are good
preimages of $1+c\alpha_0, 1+d\beta_0$ inside
$S((\overline{G/N};\tilde{\sigma},\tilde{\tau},\prec))$ and generate
a free group. Proceeding as in
Proposition~\ref{prop:reductiontonilpotent}, one sees that the
subgroup generated by $\{1+cx,1+dy\}$ inside $K(G)$ is a free group.

Now observe that if $G/N$ is torsion-free nilpotent of greater
class, there exists $M\lhd G$, such that $N\subseteq M$ and $G/M$ is
a torsion free nilpotent of class 2. Choosing an adequate
transversal for $G/M$, the foregoing argument implies the result for
$K(G)$.
\end{proof}

\section*{Acknowledgments}

I would like to thank  Jairo Z. Gon\c{c}alves and Vitor O. Ferreira
with whom I have maintained regular discussions during all stages of
this work that have proved most helpful.

\bibliographystyle{plain}
\bibliography{grupitosbuenos}

\def\cprime{$'$} \def\cprime{$'$} \def\cprime{$'$} \def\cprime{$'$}
  \def\cprime{$'$} \def\cprime{$'$} \def\acento{\'a}
\begin{thebibliography}{10}

\bibitem{AraDicks}
Pere Ara and Warren Dicks.
\newblock Universal localizations embedded in power-series rings.
\newblock {\em Forum Math.}, 19(2):365--378, 2007.

\bibitem{BellRogalskifreesubalgebrasoreextensions}
J.~P. Bell and D.~Rogalski.
\newblock Free subalgebras of quotient rings of ore extensions.
\newblock {\em Preprint available at arXiv:1101.5829v1}, 2011.

\bibitem{BludovGlassRhemtulla1}
V.~V. Bludov, M.~W. Glass, and A.~H. Rhemtulla.
\newblock Ordered groups in which all convex jumps are central.
\newblock {\em J. Korean Math. Soc.}, 40(2):225--239, 2003.

\bibitem{BludovGlassRhemtulla2}
V.~V. Bludov, M.~W. Glass, and A.~H. Rhemtulla.
\newblock On centrally orderable groups.
\newblock {\em J. Algebra}, 291:129--143, 2005.

\bibitem{chibafreegroupsinsidedivisionrings}
Katsuo Chiba.
\newblock Free subgroups and free subsemigroups of division rings.
\newblock {\em J. Algebra}, 184(2):570--574, 1996.

\bibitem{Chibafreefields}
Katsuo Chiba.
\newblock Free fields in complete skew fields and their valuations.
\newblock {\em J. Algebra}, 263:75--87, 2003.

\bibitem{Cohnskew}
P.~M. Cohn.
\newblock {\em Skew {F}ields}, volume~57 of {\em Encyclopedia of Mathematics
  and its Applications}.
\newblock Cambridge University Press, Cambridge, 1995.
\newblock Theory of general division rings.

\bibitem{Dicks&Lewin}
Warren Dicks and Jacques Lewin.
\newblock A {J}acobian conjecture for free associative algebras.
\newblock {\em Comm. Algebra}, 10(12):1285--1306, 1982.

\bibitem{LichtmanOnembeddingofgroupringsofresidually}
A.~Eizenbud and A.~I. Lichtman.
\newblock On embedding of group rings of residually torsion free nilpotent
  groups into skew fields.
\newblock {\em Trans. Amer. Math. Soc.}, 299(1):373--386, 1987.

\bibitem{FigueiredoGoncalvesShirvani}
L.~M.~V. Figueiredo, J.~Z. Gon{\c{c}}alves, and M.~Shirvani.
\newblock Free group algebras in certain division rings.
\newblock {\em J. Algebra}, 185(2):298--313, 1996.

\bibitem{Foxembedding}
Ralph~H. Fox.
\newblock Free differential calculus. {I}. {D}erivation in the free group ring.
\newblock {\em Ann. of Math. (2)}, 57:547--560, 1953.

\bibitem{Fuchs}
L.~Fuchs.
\newblock {\em Partially {O}rdered {A}lgebraic {S}ystems}.
\newblock Pergamon Press, Oxford, 1963.

\bibitem{Goncalvesfreegroupsinsubnormal}
Jairo~Z. Gon{\c{c}}alves.
\newblock Free groups in subnormal subgroups and the residual nilpotence of the
  group of units of group rings.
\newblock {\em Canad. Math. Bull.}, 27(3):365--370, 1984.

\bibitem{GoncalvesMandelShirvaniQuaternions}
Jairo~Z. Gon{\c{c}}alves, Arnaldo Mandel, and Mazi Shirvani.
\newblock Free products of units in algebras. {I}. {Q}uaternion algebras.
\newblock {\em J. Algebra}, 214(1):301--316, 1999.

\bibitem{GoncalvesShirvaniSurvey}
J.~Gon\c{c}alves and M.~Shirvani.
\newblock A survey on free objects in division rings and in division ring with
  an involution.
\newblock {\em To appear in Comm. Algebra}.

\bibitem{GoncalvesShirvani}
J.~Gon\c{c}alves and M.~Shirvani.
\newblock On free group algebras in division rings with uncountable center.
\newblock {\em Proc. Amer. Math. Soc.}, 124(3):685--687, 1996.

\bibitem{Hughes}
Ian Hughes.
\newblock Division rings of fractions for group rings.
\newblock {\em Comm. Pure Appl. Math.}, 23:181--188, 1970.

\bibitem{Hughes2}
Ian Hughes.
\newblock Division rings of fractions for group rings. {II}.
\newblock {\em Comm. Pure Appl. Math.}, 25:127--131, 1972.

\bibitem{Lam2}
T.~Y. Lam.
\newblock {\em Lectures on {M}odules and {R}ings}, volume 189 of {\em Graduate
  Texts in Mathematics}.
\newblock Springer-Verlag, New York, 1999.

\bibitem{Lichtmanonsubgroupsof}
A.~Lichtman.
\newblock On subgroups of the multiplicative group of skew fields.
\newblock {\em Proc. Amer. Math. Soc.}, 63(1):15--16, 1977.

\bibitem{Lichtmanfreesubalgebrasenvelopingalgebras}
A.~I. Lichtman.
\newblock Free subalgebras in division rings generated by universal enveloping
  algebras.
\newblock {\em Algebra Coll.}, 6(2):145--153, 1999.

\bibitem{Lichtmanuniversalfields}
A.~I. Lichtman.
\newblock On universal fields of fractions for free algebras.
\newblock {\em J. Algebra}, 231(2):652--676, 2000.

\bibitem{LongobardiMajRhemtullanofree}
P.~Longobardi, M.~Maj, and A.~H. Rhemtulla.
\newblock Groups with no free subsemigroups.
\newblock {\em Trans. Amer. Math. Soc.}, 347(4):1419--1427, 1995.

\bibitem{Makar1}
L.~Makar-Limanov.
\newblock The skew field of fractions of the {W}eyl algebra contains a free
  noncommutative subalgebra.
\newblock {\em Comm. Algebra}, 11(17):2003--2006, 1983.

\bibitem{Makar2}
L.~Makar-Limanov.
\newblock On group rings of nilpotent groups.
\newblock {\em Israel J. Math.}, 48(2-3):244--248, 1984.

\bibitem{MakarMalcolmson91}
L.~Makar-Limanov and P.~Malcolmson.
\newblock Free subalgebras of enveloping fields.
\newblock {\em Proc. Amer. Math. Soc.}, 111(2):315--322, 1991.

\bibitem{Makaronfreesubobjects}
Leonid Makar-Limanov.
\newblock On free subobjects of skew fields.
\newblock In {\em Methods in ring theory (Antwerp, 1983)}, volume 129 of {\em
  NATO Adv. Sci. Inst. Ser. C Math. Phys. Sci.}, pages 281--285. Reidel,
  Dordrecht, 1984.

\bibitem{Malcev}
A.~I. Mal{\cprime}cev.
\newblock On the embedding of group algebras in division algebras.
\newblock {\em Doklady Akad. Nauk SSSR (N.S.)}, 60:1499--1501, 1948.

\bibitem{Neumann}
B.~H. Neumann.
\newblock On ordered division rings.
\newblock {\em Trans. Amer. Math. Soc.}, 66:202--252, 1949.

\bibitem{Passman1}
Donald~S. Passman.
\newblock {\em Infinite {C}rossed {P}roducts}, volume 135 of {\em Pure and
  Applied Mathematics}.
\newblock Academic Press Inc., Boston, MA, 1989.

\bibitem{RivasOnspacesofConradorderings}
C.~Rivas.
\newblock On spaces of conradian group orderings.
\newblock {\em J. Group Theory}, 13:337--353, 2010.

\bibitem{Rosenblattinvariantmeasures}
J.~M. Rosenblatt.
\newblock Invariant measures and growth conditions.
\newblock {\em Trans. Amer. Math. Soc.}, 193:33--52, 1974.

\bibitem{SmirnovInfrainvariantsubgroups}
D.~M. Smirnov.
\newblock Infrainvariant subgroups.
\newblock {\em Ivanov. Gos. Ped. Inst. U\v c. Zap. Fiz.-Mat. Nauki}, 4:92--96,
  1953.

\bibitem{ZenkovMedvedevDlabgroups}
A.~V. Zenkov and N.~Yu. Medvedev.
\newblock Dlab groups.
\newblock {\em Algebra and Logic}, 38(5):289--298, 1999.

\end{thebibliography}

\end{document}